\documentclass{article}
\usepackage{mystyle}
\usepackage{mathtools}

\usepackage[
backend=biber,
style=alphabetic,
sorting=ynt
]{biblatex}
\DeclareMathOperator{\length}{length}

\DeclareMathOperator{\supp}{supp}
\DeclareMathOperator{\ch}{char}

\DeclareMathOperator{\coh}{\textbf{Coh}}

\DeclareMathOperator{\red}{red}

\renewcommand{\O}{O}

\newcommand{\A}{\mathbf{A}}

\renewcommand{\phi}{\varphi}

\newcommand*{\shext}{\mathcal{E}\kern -.5pt xt}
\newcommand*{\shtor}{\mathcal{T}\kern -.5pt or}

\title{Asymptotic Cohomological Functions and Volumes on Proper Schemes}
\author{Wenqi Li}
\begin{document}

\maketitle

\begin{abstract}
    We study the behavior of asymptotic cohomological functions on proper schemes over a field. We prove a formula for how asymptotic cohomological functions behave under perturbation by an object in the bounded derived category. Using the same method, we also prove that the volume of a line bundle on any proper scheme over a field always exists as a limit.
\end{abstract}

\section{Introduction}

The volume of a line bundle $\cur{L}$ on a proper scheme $X$ of dimension $d$ over a field, defined as
\[
\operatorname{vol}_X(\cur{L}) = \limsup_{m \to \infty} \frac{\dim H^0(X, \cur{L}^{\otimes m})}{m^d/d!},
\]
is an invariant that measures the asymptotic growth of global sections of $\cur{L}^{\otimes m}$ as $m$ goes to infinity. Analogously, one can define
\[
\widehat{h}^i(X, \cur{L}) = \limsup_{m \to \infty} \frac{\dim H^i(X, \cur{L}^{\otimes m})}{m^d/d!}
\]
for $i > 0$. This is defined and studied in \cite{alex} when $X$ is a projective variety (the word variety will mean a separated integral scheme of finite type over a field). More generally, if $\cur{F}$ is a coherent sheaves on $X$, we can consider the quantity
\[
\limsup_{m \to \infty} \frac{\dim H^i(X, \cur{F} \otimes \cur{L}^{\otimes m})}{m^d/d!}.
\]
We will refer to these quantities as higher asymptotic cohomological functions.

A fundamental question about the volume and the higher asymptotic cohomological functions is whether the defining limsups exist as limits. It is known classically that the volume is a limit when the scheme $X$ is a non-singular variety. It is proven in \cite[Theorem 10.7]{Cutkoskygenred} that the volume of $\cur{L}$ exists as a limit if $X$ is generically reduced. More recently, \cite{Cutnunez} and \cite{Nunez} proved that the volume exists as a limit if $X$ is a codimension-$1$ closed subscheme of a non-singular projective variety, or the ideal sheaf of nilpotents of $X$ is square-zero. In these cases, the scheme $X$ is still close to being reduced. On the other hand, Cutkosky (\cite[Theorem 10.3]{Cutkoskygenred}) showed that when $X$ is not generically reduced, the volume of a general graded linear series can fail to exist as a limit. 

For higher cohomological functions, the ``limsup versus limit'' problem is not known even for projective varieties, except in a few special cases. For example, for toric varieties, \cite{Hering2006} proved $\widehat{h}^i$ exists as limits for toric divisors. In this paper, we will prove in Theorem \ref{limit} and Theorem \ref{singlelimit} that the higher cohomological functions on $X$ exist as limits whenever they do on the reduced subscheme $X_{\red}$. In particular, we solve the ``limsup versus limit'' problem for volumes of line bundles in general:
\begin{Theorem}\label{volumelimit}
    Let $X$ be a proper scheme over a field $k$, and let $\cur{L}$ be a line bundle on $X$. The volume of $\cur{L}$ exists as a limit.
\end{Theorem}

We also prove a general formula comparing 
\[\widehat{h}^i(X, \cur{L}) \text{ and } \limsup_{m \to \infty} \frac{\dim H^i(X, \cur{F}^{\bullet} \otimes \cur{L}^{\otimes m})}{m^d/d!}\] where $\cur{F}^{\bullet}$ is a bounded complex of coherent sheaves, for a general proper $X$. When $X$ is a projective variety and $\cur{F}$ is a coherent sheaf, K{\"u}ronya showed (\cite[Proposition 2.5]{alex}) the following formula
\begin{equation}\label{basicformula}
\limsup_{m \to \infty} \frac{\dim H^i(X, \cur{F} \otimes \cur{L}^{\otimes m})}{m^d/d!} = \rank{\cur{F}} \cdot \widehat{h}^i(X, \cur{L}).
\end{equation}
When $X$ is a proper (possibly non-reduced) scheme over a field $k$, it is observed in \cite[Lemma 3.2.3]{BurgosGil2020} that
\[
\limsup_{m \to \infty} \frac{h^i(X, \cur{F} \otimes \cur{L}^{\otimes m})}{m^d/d!} \leq \rank{\cur{F}} \cdot \widehat{h}^i(X_{\red}, \cur{L}_{\red}).
\]
where $\cur{L}_{\red}$ denotes the restriction of $\cur{L}$ to $X_{\red}$. In this paper, we will prove that this is in fact an equality for any proper $X$ over an arbitrary field $k$, and generalize the formula to the case where $\cur{F}$ is an object in the bounded derived category $\cur{D}^b_{\coh}(\cur{O}_X)$ of $X$. More precisely, we will prove the following theorem

\begin{Theorem}\label{formula}
    Let $X$ be an irreducible proper scheme of dimension $d$ over a field $k$, $\cur{L}$ a line bundle on $X$, and $\cur{F}^{\bullet}$ a bounded complex of coherent $\cur{O}_X$-modules. We have
    \[
    \limsup_{m \to \infty} \frac{h^i(X, \cur{F}^{\bullet} \otimes \cur{L}^{\otimes m})}{m^d/d!} = \sum_{s+t=i} \rank{\cur{H}^t(\cur{F}}^{\bullet}) \cdot \widehat{h}^s(X_{\red}, \cur{L}_{\red}).
    \]
\end{Theorem}
We will also prove a formula (Theorem \ref{reducibleformula}) when $X$ is possibly reducible, and therefore completing the picture for any proper $X$ over a field.

This paper is organized as follows. Sections \ref{section2} and \ref{section3} are devoted to generalizing the formula \eqref{basicformula} to the case of a proper scheme $X$ and $\cur{F}^{\bullet} \in \cur{D}^b_{\coh}(\cur{O}_X)$. In Section \ref{section2} we generalize formula \eqref{basicformula} from a single coherent sheaf to a bounded complex of coherent sheaves, and in Section \ref{section3} we remove the assumptions that $X$ is reduced and irreducible, treating the characteristic $0$ and characteristic $p$ cases separately. In Section \ref{section4}, we use the methods developed in previous sections to solve the ``limsup versus limit'' problem for volumes of line bundles, and reduce this problem for higher cohomological functions to the reduced case.

\section{Perturbation by a Bounded Complex}\label{section2}

Let $X$ be a proper variety of dimension $d$ over a field $k$. Let $\cur{L}$ be a line bundle on $X$. For $0 \leq i \leq d$, the quantity
\[
\widehat{h}^i(X, \cur{L}) = \limsup_{m \to \infty} \frac{\dim H^i(X, \cur{L}^{\otimes m})}{m^d/d!}
\]
has been studied in \cite{alex}. The same paper proves that if $\cur{F}$ is a coherent sheaf on $X$, then
\begin{equation}\label{sheafperturb}
\limsup_{m \to \infty} \frac{\dim H^i(X, \cur{F} \otimes \cur{L}^{\otimes m})}{m^d/d!} = \rank{\cur{F}} \cdot \widehat{h}^i(X, \cur{L}) 
\end{equation}

Our goal in this section is to generalize this formula to the case where $\cur{F}^{\bullet}$ is a bounded complex of coherent sheaves. We will write $h^i(-)$ for $\dim H^i(-)$.

\begin{Lemma}\label{suppbd}
    Let $X$ be a proper scheme over a field $k$. If $\cur{F}^{\bullet}$ is in $\cur{D}_{\coh}^b(\cur{O}_X)$ and $\dim \supp F^{\bullet} \leq d$, then
    \[
    h^i(X, \cur{F}^{\bullet} \otimes \cur{L}^{\otimes m}) = O(m^d).
    \] 
\end{Lemma}
\begin{proof}
    This is well-known when $\cur{F}^{\bullet}$ is a single sheaf (see \cite[Example 1.2.33]{lazarsfeld2004positivity}). For complexes, we have the spectral sequence
    \begin{equation}\label{hypcohseq}
    E_2^{s,t} = H^s(X, \cur{H}^t(\cur{F}^{\bullet})) \Rightarrow H^{s+t}(X, \cur{F}^{\bullet})
    \end{equation}
    After tensoring by $\cur{L}^{\otimes m}$, it follows immediately that 
    \[
    h^i(X, \cur{F}^{\bullet} \otimes \cur{L}^{\otimes m}) \leq \sum_{s+t = i} h^s(X, \cur{H}^t(\cur{F}^{\bullet}) \otimes \cur{L}^{\otimes m}).
    \]
    Note that the sum is finite by the boundedness of $\cur{F}^{\bullet}$. Each term in the sum is $O(m^d)$, so the lemma follows.
\end{proof}

\begin{Lemma}\label{suppdiff}
    Let $X$ be a proper scheme over $k$ of dimension $d$, $\cur{L}$ a line bundle on $X$, and
    \[
    \cur{A}^{\bullet} \to \cur{B}^{\bullet} \to \cur{C}^{\bullet} \to \cur{A}^{\bullet}[1]
    \]
    a distinguished triangle in $\cur{D}^b_{\coh}(\cur{O}_X)$. If $\dim \supp \cur{A}^{\bullet} \leq d - 1$, then 
    \[
    \limsup_{m \to \infty} \frac{h^i(X, \cur{B}^{\bullet} \otimes \cur{L}^{\otimes m})}{m^d/d!} = \limsup_{m \to \infty} \frac{h^i(X, \cur{C}^{\bullet} \otimes \cur{L}^{\otimes m})}{m^d/d!}. 
    \]
\end{Lemma}

\begin{proof}
    Tensoring the given distinguished triangle by $\cur{L}^{\otimes m}$ produces a distinguished triangle
    \[
    \cur{A}^{\bullet} \otimes \cur{L}^{\otimes m} \to \cur{B}^{\bullet} \otimes \cur{L}^{\otimes m} \to \cur{C}^{\bullet} \otimes \cur{L}^{\otimes m} \to \cur{A}^{\bullet} \otimes \cur{L}^{\otimes m}[1].
    \]
    We have the long exact sequence
    \[
    \cdots \to H^i(\cur{A}^{\bullet} \otimes \cur{L}^{\otimes m}) \to H^i(\cur{B}^{\bullet} \otimes \cur{L}^{\otimes m}) \to H^i(\cur{C}^{\bullet} \otimes \cur{L}^{\otimes m}) \to 
    H^{i+1}(\cur{A}^{\bullet} \otimes \cur{L}^{\otimes m}) \to \cdots
    \]
    which implies
    \[
    \left|h^i(\cur{B}^{\bullet} \otimes \cur{L}^{\otimes m}) - h^i(\cur{C}^{\bullet} \otimes \cur{L}^{\otimes m})\right| \leq h^i(\cur{A}^{\bullet} \otimes \cur{L}^{\otimes m}) + h^{i+1}(\cur{A}^{\bullet} \otimes \cur{L}^{\otimes m}).
    \]
    By Lemma \ref{suppbd}, the right side is $O(m^{d-1})$, so after dividing by $m^n$ and taking $\limsup$, we obtain the desired equality. 
\end{proof}

\begin{Lemma}\label{extendfrompoint}
    Let $X$ be a noetherian scheme. Let $K, L \in \cur{D}^b_{\coh}(\cur{O}_X)$ and let $x$ be a point. If 
    \[
    K_x \cong L_x \text{ in } \cur{D}^b_{\coh}(\cur{O}_{X, x}),
    \]
    then there exists an open neighborhood $U$ of $x$ such that 
    \[
    K|_U \cong L|_U \text{ in } \cur{D}^b_{\coh}(\cur{O}_{U}).
    \]
\end{Lemma}
\begin{proof}
    Let $U=\spec A$ be an affine open neighborhood of $x$, corresponding to a prime $\fk{p} \subset A$. Then $K|_U$ and $L|_U$ can be represented by bounded complexes $M^{\bullet}$ and $N^{\bullet}$ of finitely generated $A$-modules. We may assume there is a quasi-isomorphism of bounded complexes of $A_{\fk{p}}$-modules
    \[
    \phi_{\fk{p}}: M_{\fk{p}}^{\bullet} \to N_{\fk{p}}^{\bullet}.
    \]
    By assumption $M_{\fk{p}}^{\bullet}$ and $N_{\fk{p}}^{\bullet}$ have finitely many terms, and each term is finitely generated, so we may choose $f \notin \fk{p}$ so that we can define $\phi_f: M_f^{\bullet} \to N_f^{\bullet}$ by the same formula as $\phi_{\fk{p}}$ on generators. Then $\phi_f$ represents a morphism $g_V: K|_V \to L|_V$ where $V = \spec A_f$. The cone of $g_V$ is an object in $\cur{D}^b_{\coh}(\cur{O}_V)$, so its support is a closed subset $Z$ of $V$, but it does not contain $x$ by construction. Hence $g_V$ is an isomorphism on $V - Z$, which is open in $X$. 
\end{proof}

\begin{Theorem}\label{redformula}
    Let $X$ be a proper variety over $k$ of dimension $n$, $\cur{L}$ a line bundle on $X$, and $\cur{F}^{\bullet}$ a bounded complex of coherent $\cur{O}_X$-modules. We have 
    \[
    \limsup_{m \to \infty} \frac{\dim H^i(X, \cur{F}^{\bullet} \otimes \cur{L}^{\otimes m})}{m^n/n!} = \sum_{s+t = i} \rank{\cur{H}^t(\cur{F}^{\bullet})} \cdot \widehat{h}^s(X, \cur{L}).
    \]
\end{Theorem}

\begin{proof}
    Using the spectral sequence \eqref{hypcohseq} and tensoring by $\cur{L}^{\otimes m}$, it follows immediately that 
    \[
    h^i(X, \cur{F}^{\bullet} \otimes \cur{L}^{\otimes m}) \leq \sum_{s+t = i} h^s(X, \cur{H}^t(\cur{F}^{\bullet}) \otimes \cur{L}^{\otimes m}).
    \]
    Note that the sum is finite by the boundedness of $\cur{F}^{\bullet}$. For simplicity set 
    \[
    \alpha_i = \limsup_{m \to \infty} \frac{\dim H^i(X, \cur{F}^{\bullet} \otimes \cur{L}^{\otimes m})}{m^n/n!}.
    \]
    Using \eqref{sheafperturb}, we see that
    \[
    \alpha_i \leq \sum_{s+t = i} \rank{\cur{H}^t(\cur{F}^{\bullet})} \cdot \widehat{h}^s(X, \cur{L}).
    \]

    We need to prove the reverse inequality. First notice that if $\cur{F}^{\bullet}$ is of the form
    \[
    \cur{F}^{\bullet} \cong \bigoplus \cur{H}^i(\cur{F}^{\bullet})[-i]
    \]
    then all differentials in the spectral sequence \eqref{hypcohseq} vanish, so the desired equality follows. 
    
    To deal with the general case, let $\eta$ be the generic point of $X$. Then $\cur{F}^{\bullet}_{\eta}$ is a complex of $K(X)$-vector spaces, so 
    \[
    \cur{F}^{\bullet}_{\eta} \cong \bigoplus_i \cur{H}^i(\cur{F}^{\bullet})_{\eta}[-i] \text{ in } \cur{D}^b(K(X)).
    \]
    Since $\cur{F}^{\bullet}$ is in $\cur{D}^b_{\coh}(\cur{O}_X)$, by Lemma \ref{extendfrompoint}, there exists a dense open set $U$ such that 
    \begin{equation}\label{splitiso}
    \cur{F}^{\bullet}|_U \cong \bigoplus_i \cur{H}^i(\cur{F}^{\bullet})|_U[-i] \text{ in } \cur{D}_{\coh}^b(\cur{O}_U).
    \end{equation}
    Let $\cur{I}$ be an ideal sheaf of $X - U$, which is a closed subset of lower dimension. Via Deligne's formula (\cite[\href{https://stacks.math.columbia.edu/tag/0G2H}{Lemma 0G2H}]{stacks-project})
    \[
    \Hom_U\left(\cur{F}^{\bullet}|_U, \bigoplus_i \cur{H}^i(\cur{F}^{\bullet})|_U[-i]\right) \cong \varinjlim_n \Hom_X\left(\cur{I}^n \cur{F}^{\bullet}, \bigoplus_i \cur{H}^i(\cur{F}^{\bullet})[-i]\right)
    \]
    the isomorphism \eqref{splitiso} corresponds to some morphism $\cur{I}^k \cur{F}^{\bullet} \to \bigoplus_i \cur{H}^i(\cur{F}^{\bullet})[-i]$ whose cone is supported on $X - U$. Also, the cone of the inclusion $\cur{I}^k \cur{F}^{\bullet} \to \cur{F}^{\bullet}$ is supported on $X - U$. Hence, by Lemma \ref{suppdiff}, we see that 
    \[
    \limsup_{m \to \infty} \frac{\dim H^i(X, \cur{F}^{\bullet} \otimes \cur{L}^{\otimes m})}{m^n/n!} = \limsup_{m \to \infty} \frac{\dim H^i(X, (\bigoplus_i \cur{H}^i(\cur{F}^{\bullet})[-i]) \otimes \cur{L}^{\otimes m})}{m^n/n!}
    \]
    so we are back in the split case discussed above.
\end{proof}

\section{Asymptotic Functions on Non-reduced Schemes}\label{section3}

We now generalize to the case when $X$ is non-reduced. Recall that for a coherent sheaf $\cur{F}$ on an irreducible (not necessarily reduced) scheme $X$ with generic point $\eta$, the rank of $\cur{F}$ is defined as the length of $\cur{F}_{\eta}$ as an $\cur{O}_{X, \eta}$-module. This definition of the rank is the same as in \cite[Definition 1.1.23]{lazarsfeld2004positivity}, but it conflicts with the usual notion of rank for locally free sheaves (e.g. in our definition the rank of $\cur{O}_X$ is $\length_{\cur{O}_{X, \eta}}(\cur{O}_{X, \eta})$, which may not be $1$.) However, we will only be using the meaning of the word ``rank'' introduced here, so there should be no possibility for confusion. 

\subsection{The general case in characteristic $0$}

Let $X$ be a proper irreducible scheme over a field $k$. We will denote its reduced subscheme by $X_{\red}$, and for any sheaf $\cur{F}$ we will denote the restriction to $X_{\red}$ by $\cur{F}_{\red}$. Let $\cur{I}$ be its ideal sheaf of nilpotents. Let $\cur{L}$ be a line bundle on $X$, and let $\cur{F}$ be a coherent $\cur{O}_X$-module. There exists a filtration
\[
0 = \cur{F}_0 \subset \cur{F}_1 \subset \cdots \subset \cur{F}
\]
such that each $\cur{F}_j/\cur{F}_{j-1}$ is of the form $i_*\cur{G}_j$ for some coherent sheaf on $X_{\red}$ by \cite[\href{https://stacks.math.columbia.edu/tag/01YF}{Lemma 01YF}]{stacks-project}. By inducting on the length of this filtration, it is easy to see that
\begin{equation}
\limsup_{m \to \infty} \frac{h^i(X, \cur{F} \otimes \cur{L}^{\otimes m})}{m^d/d!} \leq \rank{\cur{F}} \cdot \widehat{h}^i(X_{\red}, \cur{L}_{\red}).
\end{equation}
This inequality appears in \cite[Lemma 3.2.3]{BurgosGil2020}. Let $\cur{F}^{\bullet}$ be a bounded complex of coherent $\cur{O}_X$-modules. Combined with the spectral sequence \eqref{hypcohseq}, we obtain
\begin{equation}\label{easyineq}
     \limsup_{m \to \infty} \frac{h^i(X, \cur{F}^{\bullet} \otimes \cur{L}^{\otimes m})}{m^d/d!} \leq \sum_{s+t = i} \rank{\cur{H}^t(\cur{F}^{\bullet})} \cdot \widehat{h}^s(X_{\red}, \cur{L}_{\red}).
\end{equation}
Our goal in this subsection is to prove the reverse inequality when $k$ has characteristic $0$.

\begin{Lemma}\label{shortlog}
    Let $X$ be a scheme over a field $k$ of characteristic $0$. Let $\cur{I}$ be its ideal sheaf of nilpotents. Then there is an isomorphism of sheaves of abelian groups
    \[
        \log: (1 + \cur{I})^* \to \cur{I} \quad
        (1 + a) \mapsto \sum_{k \geq 1} (-1)^{k-1}\frac{a^k}{k}
    \]
    with inverse
    \[
        \exp: \cur{I} \to (1 + \cur{I})^* \quad a \mapsto \sum_{k \geq 0} \frac{a^k}{k!}.
    \]
\end{Lemma}
\begin{proof}
    The formulas for $\log$ and $\exp$ are both well-defined since $\cur{I}$ is nilpotent and $k$ has characteristic $0$. They are inverses by formal power series computations.
\end{proof}

Since $X_{\red}$ is a proper integral scheme over $k$ and $k$ has characteristic $0$ (so it is perfect), it contains an affine dense open subset $U_{\red}$ that is smooth. Let $U = \spec A$ and $U_{\red} = \spec A/I$. Since $k \to A/I$ is smooth, by the infinitesimal lifting property, e.g. \cite[\href{https://stacks.math.columbia.edu/tag/07K4}{Lemma 07K4}]{stacks-project}, we obtain a lift $A/I \to A$ of the identity on $A/I$. Thus we have a morphism $U \to U_{\red}$ such that $U_{\red} \to U \to U_{\red}$ is the identity on $U_{\red}$.

Consider the graph morphism $\Gamma: U \to U \times X_{\red}$ where $U \to X_{\red}$ is the composition $U \to U_{\red} \to X_{\red}$. Since $X_{\red}$ is separated, the graph morphism $\Gamma$ is a closed immersion. Let $X'$ be the scheme theoretic closure of $\Gamma(U)$ in $X \times X_{\red}$. Let $b: X' \to X$ be the map given by the first projection and let $f: X' \to X_{\red}$ be the map given by the second projection.

\begin{Lemma}\label{tosplit}
    Let $X, U, X'$ and $b: X' \to X$, $f: X' \to X_{\red}$ be as above. We have
    \begin{enumerate}
        \item $b^{-1}(U) \to U$ is an isomorphism.
        \item $f|_{b^{-1}(U)}: b^{-1}(U) \to X_{\red}$ lands in $U_{\red}$ and is equal to the composition of the isomorphism $b^{-1}(U) \to U$ with the given $U \to U_{\red}$.
        \item $X'$ is proper, and therefore $b$ and $f$ are both proper. 
        \item $b_{\red}: X'_{\red} \to X_{\red}$ and $f_{\red}: X'_{\red} \to X_{\red}$ are equal as morphisms of schemes, and they are isomorphisms of schemes.
        \item Both $b$ and $f$ are finite.
        \item The composition $X'_{\red} \to X' \xrightarrow{f} X_{\red} \xrightarrow{f_{\red}^{-1}} X'_{\red}$ is the identity map.
    \end{enumerate}
\end{Lemma}
\begin{proof}\leavevmode
    \begin{enumerate}
    \item 
    We know that $b^{-1}(U)$ is $X' \cap (U \times X_{\red})$, which is the closure of $\Gamma(U)$ inside $U \times X_{\red}$, but $\Gamma(U) \to U \times X_{\red}$ is a closed immersion, so $b^{-1}(U) \to U$ is the same as $\Gamma(U) \to U$ induced by the projection which is an isomorphism. 
    
    \item By 1, $b^{-1}(U) \to X_{\red}$ is then $\Gamma(U) \hookrightarrow U \times X_{\red} \to X_{\red}$ which by construction of the graph is equal to $U \to \Gamma(U) \to U_{\red}$.

    \item We have that $X'$ is proper since it is a closed subscheme of the proper scheme $X \times X_{\red}$. Then $b, f$ are proper morphisms since $X, X_{\red}$ are separated. 

    \item By the first paragraph, the morphism $b_{\red}$ and $f_{\red}$ both restrict to the isomorphism $b_{\red}^{-1}(U_{\red}) \to U_{\red}$. Since $X_{\red}$ is separated and $b_{\red}^{-1}(U_{\red})$ is dense in $X'_{\red}$, we see that $b_{\red} = f_{\red}$. 
    
    By definition of $X'$, the map $(b, f): X' \to X \times X_{\red}$ is a closed immersion, so the induced morphism 
    \[
    (b_{\red}, f_{\red}): X'_{\red} \to X_{\red} \times X_{\red}
    \] 
    is also a closed immersion. We just showed $b_{\red} = f_{\red}$, so $(b_{\red}, f_{\red})$ factors through $\Delta: X_{\red} \to X_{\red} \times X_{\red}$, which is also a closed immersion since $X_{\red}$ is separated. Therefore $X'_{\red} \to X_{\red}$ (which is $b_{\red} = f_{\red}$) is also a closed immersion. However $b_{\red}$ maps $b_{\red}^{-1}(U_{\red})$ isomorphically to $U_{\red}$ which is a scheme theoretically dense open subset of $X_{\red}$, so we see that $b_{\red}$ must also be surjective. Thus $b_{\red}$ is an isomorphism.

    \item We proved that $b$ and $f$ are bijective on points, and they are proper, so they are finite.

    \item The composition $X'_{\red} \to X' \xrightarrow{f} X_{\red}$ is just $f_{\red}$, so the statement is clear.

    \end{enumerate}
\end{proof}

\begin{Lemma}\label{finbasic}
    Let $f: X \to Y$ be a finite morphism of irreducible schemes of finite type over $k$. Let $\cur{F}^{\bullet}$ be a complex of coherent $\cur{O}_X$-modules. Then
    \begin{enumerate}[label=\alph*)]
        \item $f_*\cur{F}^{\bullet} = Rf_*\cur{F}^{\bullet}$;
        \item $\rank{\cur{H}^q(f_*\cur{F}^{\bullet})} \geq \rank{\cur{H}^q(\cur{F}^{\bullet})}$;
        \item Let $Z$ be an integral scheme over $k$ and let $z \in Z$ be any point. Let $\cur{E}^{\bullet}$ be a bounded complex of coherent sheaves on $X \times Z$. Consider the (Cartesian) diagram
        \[\begin{tikzcd}
        	X & {X \times Z} & {X \times \spec k(z)} \\
        	Y & {Y \times Z} & {Y \times \spec k(z)} \\
        	{\spec k} & Z & {\spec k(z)}
        	\arrow["f", from=1-1, to=2-1]
        	\arrow["{p_1}"', from=1-2, to=1-1]
        	\arrow["{f \times \id}", from=1-2, to=2-2]
        	\arrow["i"', from=1-3, to=1-2]
        	\arrow["f_{k(z)}", from=1-3, to=2-3]
        	\arrow[from=2-1, to=3-1]
        	\arrow[from=2-2, to=2-1]
        	\arrow[from=2-2, to=3-2]
        	\arrow[from=2-3, to=2-2]
        	\arrow[from=2-3, to=3-3]
        	\arrow[from=3-2, to=3-1]
        	\arrow[from=3-3, to=3-2]
        \end{tikzcd}\]
        We have that $((f \times \id)_*\cur{E}^{\bullet})_z = f_{k(z), *}(\cur{E}^{\bullet}_z)$, where the subscript by $z$ denotes the fiber over $z$. 
        \item In the situation of c), if $z = \spec k \hookrightarrow Z$ is a $k$-rational point, then $(p_1^*\cur{F}^{\bullet})_z = \cur{F}^{\bullet}$.
    \end{enumerate}
\end{Lemma}
\begin{proof}\leavevmode
    \begin{enumerate}[label=\alph*)]
        \item This is \cite[\href{https://stacks.math.columbia.edu/tag/0G9R}{Lemma 0G9R}]{stacks-project}.
        \item Since $f$ is finite, the statement is implied by the fact that if $A \to B$ is a ring homomorphism, then for any $B$-module $M$, the length of $M$ viewed as an $A$-module is at least the length of $M$ viewed as an $B$-module. This is because any filtration of $M$ by $B$-submodules is also a filtration by $A$-submodules.

        \item Since $f$ is affine, we can reduce to the case where $X, Y, Z$ are all affine. So let $X = \spec B$, $Y = \spec A$, and $Z = \spec R$, where $A, B, R$ are finitely generated $k$-algebras and we are given a prime ideal $\fk{p} \subset R$. Let $M$ be a $B \otimes R$-module. Let $k(\fk{p}) = R_{\fk{p}}/\fk{p}R_{\fk{p}}$. We are trying to show the isomorphism
        \[
        (M \otimes_R k(\fk{p}))_{A \otimes k(\fk{p})} \cong M_{A \otimes R} \otimes_R k(\fk{p}),
        \]
        but this is clear since they are the same set with the same $A \otimes k(\fk{p})$-action.
        
        \item Observe that $p_1 \circ i$ is the identity on $X$ after identifying $X \times \spec k$ with $X$. The operation of taking fiber by definition is pulling back by $X \to X \times Z$. Therefore 
        \[(p_1^*\cur{F}^{\bullet})_z = i^*p_1^*\cur{F}^{\bullet} = \cur{F}^{\bullet}.\]
    \end{enumerate}
\end{proof}

\begin{Theorem}\label{ch0formula}
    Let $X$ be an irreducible proper scheme of dimension $d$ over a field $k$ of characteristic $0$, $\cur{L}$ a line bundle on $X$, and $\cur{F}^{\bullet}$ a bounded complex of coherent $\cur{O}_X$-modules. We have
    \[
    \limsup_{m \to \infty} \frac{h^i(X, \cur{F}^{\bullet} \otimes \cur{L}^{\otimes m})}{m^d/d!} = \sum_{s+t=i} \rank{\cur{H}^t(\cur{F}}^{\bullet}) \cdot \widehat{h}^s(X_{\red}, \cur{L}_{\red}).
    \]
\end{Theorem}
\begin{proof}
First, assume there is some morphism $\pi: X \to X_{\red}$ such that $X_{\red} \xrightarrow{i} X \xrightarrow{\pi} X_{\red}$ is the identity. Notice that $\pi$ is finite because it is proper and bijective on points.

We have the short exact sequence
\[
0 \to (1 + \cur{I})^* \to \cur{O}_X^* \to \cur{O}_{X_{\red}}^* \to 0.\]
Using Lemma \ref{shortlog} we get a long exact sequence
\[
\cdots \to H^1(X, \cur{I}) \to \Pic X \to \Pic X_{\red} \to H^2(X, \cur{I}) \to \cdots.
\]
Both $\cur{L}$ and $\pi^*\cur{L}_{\red}$ map to $\cur{L}_{\red}$ under the map $\Pic X \to \Pic X_{\red}$. Fix a (finite) open cover $\{U_i\}$ over which $\cur{L}$ is trivialized. We can pick representatives $u_{ij} \in \cur{O}_X^*(U_{ij})$ such that $\{\ol{u_{ij}} \in \cur{O}_{X_{\red}}^*(U_{ij})\}$ as a class in $H^1(X_{\red}, \cur{O}_{X_{\red}}^*)$ represents $\cur{L}_{\red}$, and $\{u_{ij} \in \cur{O}_X^*(U_{ij})\}$ represents $\pi^*\cur{L}_{\red}$. Then from the exact sequence we see that there exists some $\{\xi_{ij} \in \cur{I}(U_{ij})\} \in H^1(X_{\red}, \cur{I})$ such that $\{u_{ij}\exp(\xi_{ij}) \in \cur{O}_X^*(U_{ij})\}$ represents $\cur{L}$.

Let $\cur{F}^{\bullet}$ be a bounded complex of coherent sheaves on $X$. Define
\[
\cur{E}^{\bullet}_m = \pi_*(\cur{F}^{\bullet} \otimes \cur{L}^{\otimes m}) \otimes \cur{L}_{\red}^{\otimes -m}.
\]
By the projection formula this is also 
\[
\cur{E}^{\bullet}_m = \pi_*(\cur{F}^{\bullet} \otimes \cur{L}^{\otimes m} \otimes \pi^*\cur{L}_{\red}^{\otimes -m}).
\]
By the analysis above we see that the line bundle $\cur{L}^{\otimes m} \otimes \pi^*\cur{L}_{\red}^{\otimes -m}$ is represented by
\[
\{\exp(\xi_{ij})^m \in (1+\cur{I})^*(U_{ij})\} = \{\exp(m\xi_{ij}) \in (1+\cur{I})^*(U_{ij})\}.
\]
Thus, we may define a line bundle $\cur{M}$ on the scheme $X \times \A^1_k$ by 
\begin{equation}\label{linebundle}
\{\exp(t \otimes \xi_{ij}) \in (k[t] \otimes \cur{O}_X(U_{ij}))^*\} \in H^1(\cur{O}^*_{X \times \A^1_k})
\end{equation}
whose fiber over $m \in \A_k^1$ is $\cur{L}^{\otimes m} \otimes \pi^*\cur{L}_{\red}^{\otimes -m}$. Define 
\[
\cur{E}^{\bullet} = (\pi \times \id)_*(p_1^*\cur{F}^{\bullet} \otimes \cur{M}).
\]
By Lemma \ref{finbasic} part c) and d), the fiber of $\cur{E}^{\bullet}$ over $m \in \A_k^1$ is $\cur{E}_m^{\bullet}$ as previously defined, and the fiber $\cur{E}_{\eta}^{\bullet}$ over the generic point $\eta \in \A_k^1$ is $\pi_*(\cur{F}^{\bullet} \otimes \cur{M}_{\eta})$.

Let $X_{\red, \eta}$ and $\cur{L}_{\red, \eta}$ be the base changes of $X_{\red}$ and $\cur{L}_{\red}$ to $k(\A_k^1) = k(t)$. We know by Theorem \ref{redformula} that
\[
\limsup_{m \to \infty} \frac{h^i(X_{\red, \eta}, \cur{E}_{\eta}^{\bullet} \otimes \cur{L}_{\red, \eta}^{\otimes m})}{m^d/d!} = \sum_{s+t=i} \rank{\cur{H}^t(\cur{E}_{\eta}^{\bullet})} \cdot \widehat{h}^s(X_{\red, \eta}, \cur{L}_{\red, \eta}).\]
because $X_{\red, \eta}$ is still reduced. Also
\[
H^i(X_{\red}, \cur{L}_{\red}^{\otimes m}) \otimes_k k(t) \cong H^i(X_{\red, \eta}, \cur{L}_{\red, \eta}^{\otimes m})
\]
and Lemma \ref{finbasic} part b) implies
\[
\rank \cur{H}^t(\cur{E}^{\bullet}_{\eta}) \geq \rank \cur{H}^t(\cur{F}^{\bullet} \otimes \cur{M}_{\eta}) = \rank \cur{H}^t(\cur{F}^{\bullet}).
\]
In fact since $\pi$ is the closed immersion, here we have equality, but the inequality is enough for our purpose. Using this we have 
\[
\sum_{s+t=i} \rank{\cur{H}^t(\cur{E}_{\eta}^{\bullet})} \cdot \widehat{h}^s(X_{\red, \eta}, \cur{L}_{\red, \eta}) \geq \sum_{s+t=i} \rank{\cur{H}^t(\cur{F}^{\bullet})} \cdot \widehat{h}^s(X_{\red}, \cur{L}_{\red}).
\]
In summary,
\[
\limsup_{m \to \infty} \frac{h^i(X_{\red, \eta}, \cur{E}_{\eta}^{\bullet} \otimes \cur{L}_{\red, \eta}^{\otimes m})}{m^d/d!} \geq \sum_{s+t=i} \rank{\cur{H}^t(\cur{F}^{\bullet})} \cdot \widehat{h}^s(X_{\red}, \cur{L}_{\red}).
\]
On the other hand,
\[
\begin{split}
h^i(X_{\red}, \cur{E}^{\bullet}_m \otimes \cur{L}_{\red}^{\otimes m}) &= h^i(X_{\red}, \pi_*(\cur{F}^{\bullet}
\otimes \cur{L}^{\otimes m})) \\
&= h^i(X_{\red}, R\pi_*(\cur{F}^{\bullet}
\otimes \cur{L}^{\otimes m})) \\
&= h^i(X, \cur{F}^{\bullet} \otimes \cur{L}^{\otimes m})
\end{split}
\]
But the function 
\[s \to h^i((X_{\red} \times \A_k^1)_s, (\cur{E}^{\bullet} \otimes p_1^*\cur{L}_{\red}^{\otimes m})_s)\]
is upper semi-continuous (by \cite[III Theorem 7.7.5]{EGA}, which applies because the projections $X_{\red} \times \A_k^1 \to X_{\red}, \A_k^1$ are flat, so cohomology sheaves of $\cur{E}^{\bullet} \otimes p_1^*\cur{L}_{\red}^{\otimes m}$ are flat over $\A_k^1$), so
\[
h^i(X_{\red}, \cur{E}_{\eta}^{\bullet} \otimes \cur{L}_{\red}^{\otimes m}) \leq h^i(X_{\red}, \cur{E}^{\bullet}_m \otimes \cur{L}_{\red}^{\otimes m}).
\]
for any $m$. Hence we obtain the reverse inequality
\[
\limsup_{m \to \infty} \frac{h^i(X, \cur{F}^{\bullet} \otimes \cur{L}^{\otimes m}))}{m^d/d!} \geq \sum_{s+t=i}\rank{\cur{H}^t(\cur{F}^{\bullet})} \cdot \widehat{h}^s(X_{\red}, \cur{L}_{\red}^{\otimes m}).
\]
Combining this with \eqref{easyineq} completes the proof of the desired equality in this special case.

Now for an arbitrary $X$, by Lemma \ref{tosplit}, there is a finite morphism $b: X' \to X$ that is bijective on points, such that there exists a factorization $X'_{\red} \xrightarrow{i'} X' \to X'_{\red}$ of the identity on $X'_{\red}$. By what we have proved, we know that
\begin{equation}\label{almost}
\limsup_{m \to \infty} \frac{h^i(X', b^*\cur{F}^{\bullet} \otimes b^*\cur{L}^{\otimes m}))}{m^d/d!} = \sum_{s+t=i}\rank{\cur{H}^s(b^*\cur{F}^{\bullet})} \cdot \widehat{h}^t(X'_{\red}, (b^*\cur{L})_{\red}^{\otimes m}).
\end{equation}
We know that $X'_{\red} \cong X_{\red}$ via $b_{\red}$, so $(b^*\cur{L})_{\red} \cong \cur{L}_{\red}$, and $b$ is an isomorphism on a dense open subset, so $\rank{\cur{H}^q(b^*\cur{F}^{\bullet})} = \rank{\cur{H}^q(\cur{F}^{\bullet})}$. Thus the right side of \eqref{almost} is equal to
\[
\sum_{s+t=i}\rank{\cur{H}^t(\cur{F}^{\bullet})} \cdot \widehat{h}^s(X_{\red}, \cur{L}_{\red}^{\otimes m}).
\]
For the left side of \eqref{almost}, we have, by finiteness of $b$,
\[
H^i(X', b^*\cur{F}^{\bullet} \otimes b^*\cur{L}^{\otimes m})) = H^i(X, b_*(b^*\cur{F}^{\bullet} \otimes b^*\cur{L}^{\otimes m})) = H^i(X, b_*b^*\cur{F}^{\bullet} \otimes \cur{L}^{\otimes m})).
\]
Since $b$ is an isomorphism on a dense open subset, the cone of $\cur{F}^{\bullet} \to b_*b^*\cur{F}^{\bullet}$ is supported on a lower dimensional closed subset. Thus by Lemma \ref{suppdiff}, 
\[
\limsup_{m \to \infty} \frac{h^i(X, \cur{F}^{\bullet} \otimes \cur{L}^{\otimes m}))}{m^d/d!} = \limsup_{m \to \infty} \frac{h^i(X, b_*b^*\cur{F}^{\bullet} \otimes \cur{L}^{\otimes m}))}{m^d/d!}.
\]
Combining all the equalities finishes the proof.

\end{proof}

\subsection{The general case in characteristic $p$}

In this subsection we will prove the same formula when the field $k$ has characteristic $p$. Notice that in this case, the $\exp$ operation as in Lemma \ref{shortlog} cannot be defined, so the method in characteristic $0$ cannot apply. Instead, we will make use of the absolute Frobenius morphism when $k$ is perfect. The first lemma here is to prepare us for base changing to the perfect closure. 

\begin{Lemma}\label{toperf}
    Let $X$ be an irreducible scheme of finite type over a field $k$ of characteristic $p > 0$. Let $k'$ be a perfect closure of $k$, and let $X' = X \otimes_k k'$. Let $\cur{F}$ be a coherent sheaf on $X$, and let $\cur{F}' = \cur{F} \otimes_k k'$. Then
    \[
    \rank{\cur{F}'} = C(X) \cdot \rank{F}.
    \]
    for some constant $C(X) > 0$ depending on the scheme $X$ but independent of $\cur{F}$. Moreover, $C(X) = C(X_{\red})$.
\end{Lemma}
\begin{proof}
    We may assume $X$ is affine. So let $X = \spec A$ where $A$ is a finitely generated $k$-algebra, and $\cur{F} = \widetilde{M}$ where $M$ is a finitely generated $A$-module. By replacing $A$ with its localization at the minimal prime, we may assume $A$ is local. Let $B = A \otimes_k k'$, which is a finitely generated $k'$-algebra (in particular, it is noetherian). Since $k'/k$ is a purely inseparable extension, the morphism $X' = \spec B \to X$ induced by $A \to B = A \otimes_k k'$ is a homeomorphism (see \cite[Proposition 2.7]{Liu}). Thus $\spec B$ is still irreducible. If $\fk{p}$ is the minimal prime of $A$, then $\fk{q} = \sqrt{\fk{p}B}$ is the minimal prime of $B$. We have $\cur{F}' = \widetilde{M \otimes_k k'} = \widetilde{M \otimes_A B}$, and
    \[
    (M \otimes_A B)_{\fk{q}} = M \otimes_A B_{\fk{q}} = M \otimes_A A_{\fk{p}} \otimes_{A_{\fk{p}}} B_{\fk{q}} = M_{\fk{p}} \otimes_{A_{\fk{p}}} B_{\fk{q}}.
    \]
    The rank of $\cur{F}'$ is then the $B_{\fk{q}}$-length of $M_{\fk{p}} \otimes_{A_{\fk{p}}} B_{\fk{q}}$. By \cite[\href{https://stacks.math.columbia.edu/tag/02M1}{Lemma 02M1}]{stacks-project}, we have
    \[
    \length_{B_{\fk{q}}}(M_{\fk{p}} \otimes_{A_{\fk{p}}} B_{\fk{q}}) = \length_{A_{\fk{p}}}(M_{\fk{p}}) \length_{B_{\fk{q}}}(B_{\fk{q}}/\fk{p} B_{\fk{q}}).
    \]
    Since $\fk{q} = \sqrt{\fk{p}B}$, the module $\length_{B}(B_{\fk{q}}/\fk{p} B_{\fk{q}})$ is annihilated by a finite power of $\fk{q}B_{\fk{q}}$ which is the maximal ideal of $B_{\fk{q}}$. Since the ring $B_{\fk{q}}$ is noetherian, we conclude that $\length_{B_{\fk{q}}}(B_{\fk{q}}/\fk{p} B_{\fk{q}})$ is finite. This is the multiplier $C(X)$, which is non-zero. To see the last claim of the lemma, observe that $X_{\red} = \spec A/{\fk{p}}$, and following the argument above we have
    \[
    C(X_{\red}) = \length_{A/\fk{q} \otimes k'}(A/\fk{q} \otimes k').
    \]
    The equality now follows from the isomorphism
    \[
    B_{\fk{q}}/\fk{p} B_{\fk{q}} = (A \otimes k')/\fk{p}(A \otimes k') \cong (A/\fk{p}) \otimes k',
    \]
    and the fact that computing length over $B_{\fk{q}}$ is the same as computing length over $B_{\fk{q}}/\fk{p}B_{\fk{q}}$. 
    
\end{proof}

Let $X$ be a proper scheme over a field $k$ of characteristic $p > 0$. Let $\cur{I}$ be its ideal sheaf of nilpotents. Let $F: X \to X$ be the absolute Frobenius. Then for some positive integer $n$, the map of sheaves $(F^\#)^n: \cur{O}_X \to \cur{O}_X$ sending $f$ to $f^{p^n}$ kills $\cur{I}$, which means that it factors through the quotient map $\cur{O}_X \to \cur{O}_X/\cur{I}$, and the morphism factors as morphisms of schemes
\[
X \xrightarrow{\pi} X_{\red} \xrightarrow{i} X.
\]
The morphism $\pi$ has the following properties.
\begin{enumerate}
    \item The morphism $\pi$ is the identity map on the underlying topological space.
    \item If $\cur{L}$ is a line bundle on $X$, then $\pi^*\cur{L}_{\red} \cong \cur{L}^{\otimes p^n}$.
    \item The morphism $\pi$ is finite.
\end{enumerate}

\begin{Lemma}\label{frobrank}
    Let $X$ be an irreducible scheme of dimension $d$ over a perfect field $k$ of characteristic $p$. Let $\cur{F}$ be a coherent $\cur{O}_X$-module. Let $F$ be the absolute Frobenius. Then 
    \[\rank{F_*\cur{F}} = p^d\rank{\cur{F}}.\]
\end{Lemma}
\begin{proof}
    If $\eta$ is the generic point of $X$, then the dimension of $X$ is equal to the transcendence degree of $K = \kappa(\eta) = K(X_{\red})$ over $k$. By Theorem 26.2 and Theorem 26.3 in \cite{Matsumura_1987}, $K$ has a separating transcendence basis. Namely, we may choose a transcendence basis $t_1, \cdots, t_d$ such that $K$ is a finite separable extension over $k(t_1, \cdots, t_d)$. We have a diagram where all arrows are inclusions
    \[\begin{tikzcd}
	{F^\#(K)} & K \\
	{k(t_1^p, \cdots, t_d^p)} & {k(t_1, \dots, t_d)}
	\arrow[from=1-1, to=1-2]
	\arrow[from=2-1, to=1-1]
	\arrow[from=2-1, to=2-2]
	\arrow[from=2-2, to=1-2].
    \end{tikzcd}\]
    The two vertical extensions have the same degree, so the two horizontal extensions have the same degree. The degree of the bottom extension is clearly $p^d$, so $[K: F^\#(K)] = p^d$.

    The rank depends only on the stalk at the generic point, so we may reduce to the following situation: $(R, \fk{m})$ is an artin local $k$-algebra such that $R/\fk{m}$ has transcendence degree $d$ over $k$, and $M$ is a finitely generated $R$-module. Then $F_*\cur{F}$ is still given by the module $M$ whose $R$-action now is 
    \[
    r \cdot m = r^p m.
    \]
    So choose a maximal chain
    \[
    0 = M_0 \subset M_1 \subset \cdots \subset M
    \]
    Then each $M_i/M_{i-1}$ is isomorphic to $R/\fk{m}$, and $R/\fk{m}$ has length $[K:F^\#(K)]$ when viewed as a module using the new action, so we see that
    \[\rank{F_*\cur{F}} = [K:F^\#(K)]\cdot \rank{\cur{F}} = p^d\rank{\cur{F}}.\]
\end{proof}

\begin{Theorem}\label{chpformula}
    Let $X$ be an irreducible proper scheme of dimension $d$ over a field $k$ of characteristic $p$ and let $\cur{L}$ be a line bundle on $X$. Let $\cur{F}^{\bullet}$ be a complex of coherent $\cur{O}_X$-modules. We have
    \[
    \limsup_{m \to \infty} \frac{h^i(X, \cur{F}^{\bullet} \otimes \cur{L}^{\otimes m})}{m^d/d!} = \sum_{s+t = i} \rank{\cur{H}^t(\cur{F}^{\bullet})} \cdot \widehat{h}^s(X_{\red}, \cur{L}_{\red}).
    \]
\end{Theorem}
\begin{proof}
Let $k'$ be a perfect closure of $k$. Let $X'$, $\cur{L}'$,  $(\cur{F}^{\bullet})'$ and $\cur{H}^t(\cur{F}^{\bullet})'$ be the base changes to $k'$. We have
\[
H^i(X, \cur{F}^{\bullet} \otimes \cur{L}^{\otimes m}) \otimes_k k' = H^i(X', (\cur{F}^{\bullet})' \otimes (\cur{L}')^{\otimes m})
\]
and by Lemma \ref{toperf}, there is a constant $C(X)$ such that $\rank{\cur{H}^t(\cur{F}^{\bullet})'} = C(X)\rank{\cur{H}^t(\cur{F}^{\bullet})}$. Consider the diagram
\[\begin{tikzcd}
	{X \otimes k'} & X \\
	{X_{\red} \otimes k'} & {X_{\red}} \\
	{(X_{\red} \otimes k')_{\red}}
	\arrow[from=1-1, to=1-2]
	\arrow[from=2-1, to=1-1]
	\arrow[from=2-1, to=2-2]
	\arrow[from=2-2, to=1-2]
	\arrow["j"', from=3-1, to=2-1]
\end{tikzcd}\]
The square is a flat base change square, and $(X_{\red} \otimes k')_{\red} = (X')_{\red}$. If the result is known for perfect fields, then
\[
\begin{split}
h^i(X_{\red}, \cur{L}_{\red}^{\otimes m})
&= h^i((X_{\red} \otimes k', (\cur{L}_{\red} \otimes k')^{\otimes m}) \\
&= \rank{\cur{O}_{X_{red} \otimes k'}} \cdot h^i((X')_{\red}, (\cur{L}')_{\red}^{\otimes m}) \\
&= C(X_{\red}) \cdot h^i((X')_{\red}, (\cur{L}')_{\red}^{\otimes m})
\end{split}
\]
So we can compute
\[
\begin{split}
    \limsup_{m \to \infty} \frac{h^i(X, \cur{F}^{\bullet} \otimes \cur{L}^{\otimes m})}{m^d/d!} &= \limsup_{m \to \infty} \frac{h^i(X', (\cur{F}^{\bullet})' \otimes (\cur{L}')^{\otimes m})}{m^d/d!}\\ 
    &= \sum_{s+t = i} \rank{\cur{H}^t(\cur{F}^{\bullet})'} \cdot \widehat{h}^s(X'_{\red}, (\cur{L})'_{\red})) \\
    &= \sum_{s+t = i} C(X)\rank{\cur{H}^t(\cur{F}^{\bullet})} \cdot \widehat{h}^s(X'_{\red}, (\cur{L})'_{\red})) \\
    &= \sum_{s+t = i} C(X)\rank{\cur{H}^t(\cur{F}^{\bullet})} \cdot \frac{1}{C(X_{\red})} \cdot \widehat{h}^s(X_{\red}, \cur{L}_{\red}))\\
    &= \sum_{s+t = i} \rank{\cur{H}^t(\cur{F}^{\bullet})} \cdot\widehat{h}^s(X_{\red}, \cur{L}_{\red})).
\end{split}
\]
where we used $C(X) = C(X_{\red})$ by Lemma \ref{toperf}. Hence if the formula is true for $k'$, then it true for $k$, so we may assume the field $k$ is perfect.

Let $\pi: X \to X_{\red}$ be as in the discussion above, i.e. $F^n = i \circ \pi$. Then $F^n_* = i_* \circ \pi_*$, and $\rank F^n_*\cur{F} = \rank \pi_*\cur{F}$. For any sufficiently large $m$, we can write $m = qp^n + r$ where $0 \leq r < p^n$ and $q > 0$, and
\begin{equation}\label{chpcompute}
\begin{split}
h^i(X, \cur{F}^{\bullet} \otimes \cur{L}^{\otimes m}) &= h^i(X, \cur{F}^{\bullet} \otimes \cur{L}^{\otimes r} \otimes (\cur{L}^{\otimes p^n})^{\otimes q})) \\
&= h^i(X, \cur{F}^{\bullet} \otimes \cur{L}^{\otimes r} \otimes \pi^*\cur{L}_{\red}^{\otimes q}) \\
&= h^i(X_{\red}, R\pi_*(\cur{F}^{\bullet} \otimes \cur{L}^{\otimes r}) \otimes \cur{L}_{\red}^{\otimes q}) \\
&= h^i(X_{\red}, \pi_*(\cur{F}^{\bullet} \otimes \cur{L}^{\otimes r}) \otimes \cur{L}_{\red}^{\otimes q}) \\
\end{split}
\end{equation}
where the third equality is the projection formula, and the fourth equality is by Lemma \ref{finbasic} part a). Consider the subsequences:
\begin{equation}\label{subseq}
a_{q}^{(r)} = \frac{h^i(X_{\red}, \pi_*(\cur{F}^{\bullet} \otimes \cur{L}^{\otimes r}) \otimes \cur{L}_{\red}^{\otimes q})}{(qp^n + r)^d/d!}.
\end{equation}
For each $0 \leq r < p^n$, we have
\begin{equation}
\begin{split}
\limsup_{q \to \infty} \frac{h^i(X_{\red}, \pi_*(\cur{F}^{\bullet} \otimes \cur{L}^{\otimes r}) \otimes \cur{L}_{\red}^{\otimes q})}{(qp^n + r)^d/d!}
&\leq \limsup_{q \to \infty} \frac{h^i(X_{\red}, \pi_*(\cur{F}^{\bullet} \otimes \cur{L}^{\otimes r}) \otimes \cur{L}_{\red}^{\otimes q})}{(qp^n)^d/d!} \\
&= \frac{1}{p^{nd}}\sum_{s+t=i}\rank{\cur{H}^t(\pi_*\cur{F}^{\bullet})} \cdot \widehat{h}^s(X_{\red}, \cur{L}_{\red}) \\
&= \frac{1}{p^{nd}}\sum_{s+t=i}\rank{\pi_*\cur{H}^t(\cur{F}^{\bullet})} \cdot \widehat{h}^s(X_{\red}, \cur{L}_{\red}) \\
&= \sum_{s+t=i} \rank{\cur{H}^t(\cur{F}^{\bullet}}) \cdot \widehat{h}^s(X_{\red}, \cur{L}_{\red}) \\
\end{split}
\end{equation}
Here we used Theorem \ref{redformula} on the second line and Lemma \ref{frobrank} on the last line. At this point, we proved that we have finitely many subsequences $a_q^{(r)}$ which form a partition of the original sequence and satisfy
\[
\limsup_{q \to \infty} a_q^{(r)} \leq \limsup_{q \to \infty} a_q^{(0)}.
\]
Therefore 
\[
\limsup_{m \to \infty} \frac{h^i(X, \cur{F}^{\bullet} \otimes \cur{L}^{\otimes m})}{m^d/d!} = \limsup_{q \to \infty} a_q^{(0)} = \sum_{s+t=i} \rank{\cur{H}^t(\cur{F}^{\bullet}}) \cdot \widehat{h}^s(X_{\red}, \cur{L}_{\red}).
\]    
\end{proof}

Combining Theorem \ref{ch0formula} and Theorem \ref{chpformula} gives the proof of Theorem \ref{formula}.

\subsection{Reducible Schemes}

Until this point we have assumed our scheme $X$ is irreducible. In this subsection we will generalize our results to possibly reducible schemes.

\begin{Theorem}\label{reducibleformula}
    Let $X$ be a proper scheme of dimension $d$ over a field $k$. Let $Y_1, \cdots, Y_k$ be its irreducible components of dimension $d$ with the reduced scheme structure. Let $\cur{L}$ be a line bundle on $X$, and $\cur{F}^{\bullet}$ a bounded complex of coherent $\cur{O}_X$-modules. Denote the restriction of $\cur{L}$ and $\cur{F}^{\bullet}$ to $Y_j$ by $\cur{L}_j$ and $\cur{F}^{\bullet}_j$. Then
    \[
    \limsup_{m \to \infty} \frac{h^i(X, \cur{F}^{\bullet} \otimes \cur{L}^{\otimes m})}{m^d/d!} = \sum_{j=1}^k \sum_{s+t=i} \rank{\cur{H}^t(\cur{F}_j^{\bullet})} \cdot \widehat{h}^s(Y_{j}, \cur{L}_{j}).
    \]
\end{Theorem}

Here the use of the notation $\rank$ is consistent with before, except it is understood to be computed over the correct generic point: if $\cur{G}$ is a coherent sheaf on $Y_j$ with generic point $\eta_j$, then $\rank{\cur{G}}$ means $\length_{\cur{O}_{Y_j, \eta_j}}(\cur{G}_{\eta_j})$.

\begin{proof}
    Let $Y_1, \cdots, Y_k$ be irreducible components of $X$ of maximal dimension, and let $Y_{k+1}, \cdots, Y_l$ be all other irreducible components. For each $1 \leq j \leq l$, take the non-empty open subschemes $U_j \subset Y_j$ given by $U_j = Y_j - \cup_{r \neq j} Y_r$ (note that $U_j$ is open in both $Y_j$ and $X$), and let $Z_j$ be the scheme theoretic closure of $U_j$. Observe that the cone of the map 
    \begin{equation}\label{componentmap}
    \cur{F}^{\bullet} \to \bigoplus _{j=1}^l i_{Z_j,*}i_{Z_j}^*\cur{F}^{\bullet} 
    \end{equation}
    is zero on each $U_j$: when restricted to $U_j$, for $j \neq r$ the complex $i_{Z_r,*}i_{Z_r}^*\cur{F}^{\bullet}|_{U_j}$ is zero because $U_j$ is disjoint from $Z_r$, so \eqref{componentmap} becomes
    \[
    \cur{F}^{\bullet}|_{U_j} \to (i_{Z_j,*}i_{Z_j}^*\cur{F}^{\bullet})|_{U_j}
    \]
    which is an isomorphism because $U_j$ is contained in $Z_j$. Hence the cone of \eqref{componentmap} is supported on $X - \cup_j U_j$ which is a closed subset of lower dimension. By Lemma \ref{suppdiff}, we have
    \[
    \begin{split}
        \limsup_{m \to \infty} \frac{h^i(X, \cur{F}^{\bullet} \otimes \cur{L}^{\otimes m})}{m^d/d!} &= \limsup_{m \to \infty} \frac{h^i\left(X, \left(\bigoplus_j^{l} i_{Z_j,*}i_{Z_j}^*\cur{F}^{\bullet}\right) \otimes \cur{L}^{\otimes m}\right)}{m^d/d!} \\ 
        &= \sum_{j=1}^l \limsup_{m \to \infty} \frac{h^i(Z_j, i_{Z_j}^*\cur{F}^{\bullet} \otimes \cur{L}^{\otimes m})}{m^d/d!}.
    \end{split}
    \]
    By construction, $Y_j$ equipped with the reduced scheme structure is exactly $(Z_j)_{\red}$, and irreducible components of lower dimension will give $0$ in the sum by Lemma \ref{suppbd}. So by the results in the irreducible case (Theorem \ref{ch0formula} and Theorem \ref{chpformula}), we have
    \[
    \sum_{j=1}^l \limsup_{m \to \infty} \frac{h^i(Z_j, i_{Z_j}^*\cur{F}^{\bullet} \otimes \cur{L}^{\otimes m})}{m^d/d!} = \sum_{j=1}^k \sum_{s+t = i} \rank{\cur{H}^t(\cur{F}_j^{\bullet})} \cdot \widehat{h}^s(Y_j, \cur{L}_j).
    \]
\end{proof}

\section{Convergence of Volumes of Line Bundles}\label{section4}

In this section we will prove results on when the defining limsups of the volume and higher asymptotic cohomology functions of a line bundle are in fact limits. We will prove general results and then deduce Theorem \ref{volumelimit}.

\begin{Theorem}\label{limit}
    Let $X$ be a proper scheme of dimension $d$ over a field $k$, and let $\cur{L}$ be a line bundle on $X$. Let $\cur{F}^{\bullet}$ be a bounded complex of coherent $\cur{O}_X$-modules. Let $i$ be a non-negative integer. The sequence 
    \[
    \frac{h^i(X, \cur{F}^{\bullet} \otimes \cur{L}^{\otimes m})}{m^d/d!}
    \]
    converges if 
    \[
    \widehat{h}^j(X_{\red}, \cur{L}_{\red}) = \limsup_{m \to \infty} \frac{h^i(X_{\red}, \cur{L}_{\red}^{\otimes m})}{m^d/d!}
    \]
    exists as a limit for all $j \geq 0$.
\end{Theorem}

\begin{proof}
    We immediately reduce to the case where $X$ is irreducible. Again we need to treat the case of $\ch k = 0$ and $\ch k = p > 0$ separately.

    First suppose $\ch k = 0$, and assume there is a morphism $\pi: X \to X_{\red}$. As in the proof of Theorem \ref{ch0formula}, let
    \[
    \cur{E}^{\bullet} = (\pi \times \id)_*(p_1^*\cur{F}^{\bullet} \otimes \cur{M})
    \]
    where $\cur{M}$ is the line bundle defined by \eqref{linebundle}. 
    Then
    \[
    a_m \coloneqq \frac{h^i(X, \cur{F}^{\bullet} \otimes \cur{L}^{\otimes m})}{m^d/d!} = \frac{h^i(X_{\red}, \cur{E}^{\bullet}_m \otimes \cur{L}_{\red}^{\otimes m})}{m^d/d!}
    \geq \frac{h^i(X_{\red, \eta}, \cur{E}_{\eta}^{\bullet} \otimes \cur{L}_{\red, \eta}^{\otimes m})}{m^d/d!} \eqqcolon b_m.
    \]
    By assumption $b_m$ converges to the limit
    \[
    \lim_{m \to \infty} b_m = \sum_{s+t=i} \rank{\cur{H}^t(\cur{E}^{\bullet})} \cdot \widehat{h}^s(X_{\red}, \cur{L}_{\red}) \geq \sum_{s+t=i} \rank{\cur{H}^t(\cur{F}^{\bullet})} \cdot \widehat{h}^s(X_{\red}, \cur{L}_{\red}),
    \]
    and we have shown $a_m$ has the limsup:
    \[
    \limsup_{m \to \infty} a_m = \sum_{s+t=i} \rank{\cur{H}^t(\cur{F}^{\bullet})} \cdot \widehat{h}^s(X_{\red}, \cur{L}_{\red})
    \]
    but $a_m \geq b_m$, so 
    \[
    \liminf_{m \to \infty} a_m \geq \liminf_{m \to \infty} b_m = \lim_{m \to \infty} b_m \geq \limsup_{m \to \infty} a_m.
    \] 
    Thus $a_m$ converges.

    In the general case, as in the proof of Theorem \ref{ch0formula}, using Lemma \ref{tosplit} we obtain a proper birational $b: X' \to X$ such that $X'$ admits a splitting $\pi: X' \to X'_{\red}$. The cone of the natural morphism $\cur{F}^{\bullet} \to b_*b^*\cur{F}^{\bullet}$ is supported away from a dense open subset, so 
    \[
    \left|h^i(X, \cur{F}^{\bullet} \otimes \cur{L}^{\otimes m}) - h^i(X, b_*b^*\cur{F}^{\bullet} \otimes \cur{L}^{\otimes m})\right| = O(m^{d-1}).
    \]
    Since 
    \[
    h^i(X, b_*b^*\cur{F}^{\bullet} \otimes \cur{L}^{\otimes m}) = h^i(X', b^*\cur{F}^{\bullet} \otimes b^*\cur{L}),
    \]
    after dividing by $m^d$ the desired result follows from the case discussed above.

    Now suppose $\ch k = p > 0$ and $k$ is perfect. As in the proof of Lemma \ref{chpformula}, let $\pi: X \to X_{\red}$ be the morphism such that $F^n = X \xrightarrow{\pi} X_{\red} \to X$ where $F$ is the absolute Frobenius of $X$ and $n$ is a fixed large integer. For any positive integer $m > p^n$, we may write $m = qp^n + r$ where $q > 0$ and $0 \leq r < p^n$. With the same notation \eqref{subseq} as in the proof of Theorem \ref{chpformula}, under the current hypothesis we know that $a_q^{(0)}$ is convergent, and for $0 < r < p^n$
    \[
    a_q^{(r)} = \frac{h^i(X_{\red}, \pi_*(\cur{F}^{\bullet} \otimes \cur{L}^{\otimes r}) \otimes \cur{L}_{\red}^{\otimes q})}{(qp^n)^d/d!} \cdot \frac{(qp^n)^d}{(qp^n+r)^d}
    \]
    is a product of two convergent sequences and the second one has limit $1$, so it also converges to the same limit as $a_q^{(0)}$. It then follows that the sequence
    \[
    \frac{h^i(X, \cur{F}^{\bullet} \otimes \cur{L}^{\otimes m})}{m^d/d!}
    \]
    converges to the same limit.
    
\end{proof}

\begin{Theorem}\label{singlelimit}
    Let $X$ be a proper scheme of dimension $d$ over a field $k$, and let $\cur{L}$ be a line bundle on $X$. Let $\cur{F}$ be a coherent $\cur{O}_X$-modules. Let $i$ be a non-negative integer. The sequence 
    \[
    \frac{h^i(X, \cur{F} \otimes \cur{L}^{\otimes m})}{m^d/d!}
    \]
    converges if 
    \[
    \widehat{h}^i(X, \cur{L}) = \limsup_{m \to \infty} \frac{h^i(X_{\red}, \cur{L}_{\red}^{\otimes m})}{m^d/d!}
    \]
    exists as a limit.
\end{Theorem}

\begin{proof}
    The proof is the same as that of Theorem \ref{limit}, and in this special case where $\cur{F}^{\bullet}$ is now just a single sheaf $\cur{F}$, the sum
    \[
    \sum_{s+t=i} \rank{\cur{H}^t(\cur{F}^{\bullet})} \cdot \widehat{h}^s(X_{\red}, \cur{L}_{\red})
    \]
    just has single term $\rank{\cur{F}} \cdot \widehat{h}^i(X_{\red}, \cur{L}_{\red})$, so the current hypothesis guarantees that it exists as a limit.
\end{proof}

Now we can easily deduce Theorem \ref{volumelimit}.

\begin{proof}[Proof of Theorem \ref{volumelimit}]
    This follows now from Theorem \ref{singlelimit} with $\cur{F}^{\bullet} = \cur{O}_X$, and the known result that the volume of $\cur{L}_{\red}$ on $X_{\red}$ exists as a limit (e.g. \cite{Cutkoskygenred}).
\end{proof}

\printbibliography

\end{document}